\documentclass[11pt]{article}

\usepackage{amsmath, amssymb, latexsym, amsthm}
\usepackage{hyperref}
\usepackage{xcolor}
\usepackage[final]{changes}

\newtheorem{theorem}{Theorem} 
\newtheorem{prop}[theorem]{Proposition}
\newtheorem{cor}[theorem]{Corollary}

\newtheorem{lemma}{Lemma}
\newtheorem{define}{Definition}

\newcommand{\eps}{\varepsilon}

\newcommand{\cB}{\mathcal{B}}
\newcommand{\Nat}{\mathbb{N}}
\newcommand{\vol}{\mathrm{vol}}

\DeclareMathOperator\disp{disp}
\begin{document}

\title{Lower bounds on the minimal dispersion of point sets via cover-free families\footnotetext{
The work of the first and the third author has been supported by the grant P202/23/04720S of the Grant Agency of the Czech Republic.
The work of the second author has been supported by the grant 23-06815M of the Grant Agency of the Czech Republic.
}}

\author{M. Tr\"odler\thanks{Department of Mathematics, Faculty of Nuclear Sciences and Physical Engineering,
Czech Technical University in Prague, Trojanova 13, 12000 Prague, Czech Republic.
E-mail: \href{mailto:trodlmat@fjfi.cvut.cz}{\hbox{trodlmat@fjfi.cvut.cz}}},
\and J. Volec\thanks{
Department of Theoretical Computer Science, Faculty of Information Technology, Czech Technical University in Prague, Th\'akurova 9, Prague, 160 00, Czech Republic.
E-mail: \href{mailto:jan@ucw.cz}{jan@ucw.cz}},
\and and J. Vyb\'\i ral\thanks{%Corresponding author.
Department of Mathematics, Faculty of Nuclear Sciences and Physical Engineering,
Czech Technical University in Prague, Trojanova 13, 12000 Prague, Czech Republic.
E-mail: \href{mailto:jan.vybiral@fjfi.cvut.cz}{\hbox{jan.vybiral@fjfi.cvut.cz}}}
}
\maketitle

\begin{abstract}
We elaborate on the intimate connection between the largest volume of an empty axis-parallel box in a set of $n$ points from $[0,1]^d$ and cover-free families from the extremal set theory.
This connection was discovered in a recent paper of the authors.
In this work, we apply a very recent result of Michel and Scott to obtain a whole range of new lower bounds on the number of points needed
so that the largest volume of such a box is bounded by a given $\eps$.
Surprisingly, it turns out that for each of the new bounds, there is a choice of the parameters $d$ and $\eps$ such that the bound outperforms the others.
\end{abstract}

\section{Introduction}

Let $X\subset[0,1]^d$ be a (finite) set of points.
There are several ways of how to measure whether the points of $X$ are well spread.
One way, which has recently attracted the attention of many researchers, is the so-called \emph{dispersion}.
The dispersion of $X$ is the volume of the largest axis-parallel box in $[0,1]^d$ that contains no point of $X$, i.e.,
\begin{equation}\label{eq:def_disp}
\disp(X):=\sup_{B:B\cap X=\emptyset} \added{\vol(B)}\text{\deleted{$|B|$}}.
\end{equation}
Here, the supremum is taken over all the boxes $B=\prod_{i=1}^d (a_i,b_i)$, where $0\le a_i<b_i\le 1$ for all $i\in[d],$ and $\added{\vol(B)}$
\deleted{$|B|$}
stands for the (Lebesgue) volume of $B$.
The study of this notion goes back to \cite{Hlawka, Nieder} and \cite{RT96}.

A very natural extremal problem is to determine the smallest possible dispersion for a given number of points in $[0,1]^d$.
For fixed integers $d$ and $n$, we denote the minimum dispersion of an $n$-point set $X\subseteq [0,1]^d$ by
\begin{equation}\label{eq:def_disp_nd}
\disp^*(n,d):=\inf_{\substack{X\subset [0,1]^d\\|X|=n}}\disp(X).
\end{equation}

It is sometimes convenient to study this problem in the inverse setting:
given an integer $d$ and $\eps\in (0,1)$, how many points do we need to place into the $d$-dimensional unit cube (and how?) so that their dispersion is at~most~$\eps$?
%For a fixed $d \in \Nat$ and $\eps\in(0,1)$, we define the inverse function of the minimal dispersion in $[0,1]^d$ by
The extremal problem is reflected by the quantity
\begin{align}\label{eq:def_N}
N(\varepsilon,d)&:=\min\{n\in\Nat: \disp^*(n,d)\le \varepsilon\}\\
\notag &=\min\{n\in\Nat: \exists X\subset [0,1]^d\ \text{with}\ |X|=n\ \text{and}\ \disp(X)\le\varepsilon\}.
\end{align}

\medskip

As mentioned above, a number of upper and lower bounds on the quantities $\disp^*(n,d)$ and $N(\eps,d)$ have been established.
The elementary lower bound $\disp^*(n,d)\ge \frac{1}{n+1}$ was improved in \cite{Dum} to $\disp^*(n,d)\ge \frac{5}{4(n+5)}$.
As the next step, for $d\ge 2$ and $\eps \in(0,1/4)$, it was shown in~\cite{AHR} that
\begin{equation}\label{eq:AHR}
N(\varepsilon,d)\ge \frac{\log_2 d}{8\varepsilon},
\end{equation}
which seems to be the first lower bound on $N(\eps,d)$ that grows with $d$.
 
A further improvement was obtained by Bukh and Chao~\cite{BC}, who proved
\begin{equation}\label{eq:BC_disp}
\disp^*(n,d)\ge \frac{1}{e}\cdot\frac{2d}{n}\left(1-\frac{4d}{n^{1/d}}\right).
\end{equation}
This can be translated into the following lower bound on the inverse problem:
\begin{equation}\label{eq:BC_lower}
N(\eps,d)\ge \frac{1}{e}\cdot \frac{d}{\eps} \quad \text{ for every $\eps \le (8d)^{-d}$}.
\end{equation}
The bounds \eqref{eq:BC_disp} and \eqref{eq:BC_lower} are nearly optimal when $d$ is fixed,
and we study the limiting behavior for $n$ (or $1/\eps$) tending to infinity.
However, note that \eqref{eq:BC_disp} yields \eqref{eq:BC_lower} only for very small values of $\eps$.

Regarding upper bounds on $N(\eps,d)$, a series of recent papers \cite{Lit,Sosnovec, UV} have shown that 
\begin{equation}\label{eq:TVV_1}
N(\eps,d)\le \frac{C \log d\cdot\log\frac1\eps}{\eps^2}\,.
\end{equation}
In~\cite{TVV} we have shown that this bound is nearly optimal for $\eps$ being large:
\begin{theorem}[\cite{TVV}]\label{thm:OLD}
There is an absolute constant $c>0$ such that for any integer $d\ge 2$ and any real $\eps \in \left(\frac1{4\sqrt{d}}, \frac14\right)$,
it holds that
\begin{equation}\label{eq:TVV_2}
N(\eps,d) >  \frac{ c\,\log d }{ \eps^2 \cdot \log{\frac1\eps} } 
\,.
\end{equation}
\end{theorem}
On one hand, \eqref{eq:TVV_2} matches \eqref{eq:TVV_1} up to a polylogarithmic factor in $1/\eps$.
On the other hand, \eqref{eq:TVV_2} holds only for $\eps$ of the order at least $1/\sqrt{d}$.
Some further bounds for different regimes of $\varepsilon$ and $d$ were obtained in \cite{AL, Lit, LL, Mac, Rudolf},
and the dispersion of certain specific sets was studied also in \cite{BH, HKKR, Krieg, Kritz, LW, Teml,Mario}.

The aim of this work is to establish lower bounds on $N(\eps,d)$ that are valid when $\eps < 1/\sqrt{d}$.
Our main result is the following.
\begin{theorem}\label{thm:k_gen}
Fix a positive integer $k$. There exists a constant $c_k$ such that 
if $d$ is a positive integer satisfying $d\ge d^{\frac{k}{k+1}}+k$ and $\eps \in \left(0,2^{-k-2}\right)$, then the following is true.
\begin{enumerate} \item[(i)] If $\eps\ge d^{-\frac{k^2}{k+1}}$ then $N(\eps,d)\ge c_k \cdot \eps^{-\frac{k+1}{k}}$.
\item[(ii)] If $\eps <d^{-\frac{k^2}{k+1}}$ then $N(\eps,d)\ge c_k \cdot d^{\frac{k}{k+1}}\,\eps^{-1}$.
\end{enumerate}
\end{theorem}
The proof of part (i) of Theorem \ref{thm:k_gen} is given in Section \ref{sec:2} and the second part of Theorem \ref{thm:k_gen}
is demonstrated in Section \ref{sec:3}. \deleted{We postpone the comparison of the bounds given in Theorem \ref{thm:k_gen} to other bounds available in the literature to Section 4.}
%\subsection{Conclusion}\label{sec:4}
%\added{But before we come to the proof of Theorem} \ref{thm:k_gen} let us discuss the relation of Theorem \ref{thm:k_gen} to the lower (and upper) bounds on $N(\eps,d)$ available in the literature.

\medskip

Let us now discuss the relation of Theorem \ref{thm:k_gen} to the lower (and upper) bounds on $N(\eps,d)$ available in the literature.

\medskip

We start with $k=1$. Due to the restriction $\eps\ge d^{-1/2}$, we always have $\frac{\log(d)}{\log(1/\eps)}\ge 2$
and the bound in Part (i) of Theorem \ref{thm:k_gen} is inferior to \eqref{eq:TVV_2} by logarithmic factors.
Observe that using Corollary \ref{cor:N_C} together with \cite[Lemma 2.8]{AA} (which was the main ingredient of our previous work \cite{TVV}) instead of \cite[Lemma 2.2]{Michel_Scott}
would essentially recover Theorem \ref{thm:OLD} again.
Part (ii) of Theorem~\ref{thm:k_gen} for $k=1$ states that
\begin{equation}\label{eq:elongation}
		N(\eps,d) \geq 
C\,\frac{\sqrt{d}}{\eps}
\end{equation}
and it clearly outperforms (possibly up to multiplicative constants) the lower bound
of \cite{AHR}, cf. \eqref{eq:AHR}. On the other hand, \eqref{eq:BC_lower} is stronger, but note that the validity of \eqref{eq:BC_lower}
was restricted to extremely small $\eps$'s. In contrast, \eqref{eq:elongation} holds for $\eps=O(d^{-1/2})$.
Finally, let us note that \cite[Theorem 2]{BC} shows that
\[
N(\eps,d)\le \frac{c\, d^2\log(d)}{\eps}.
\]
%We can merge \eqref{eq:elongation} with Theorem~\ref{thm:OLD} as follows.
%There is an absolute constant $c>0$, such that for every $d\ge 2$ and every $0<\eps<1/4$ it holds
%\begin{equation}\label{eq:merge_1}
%N(\eps,d)\ge c\min\left(\frac{\sqrt{d}}{\eps},\frac{\log d}{\eps^2 \log(1/\eps)}\right).
%\end{equation}

When increasing the value of $k$, Theorem \ref{thm:k_gen} yields a whole series of lower bounds on $N(\eps,d)$, conveniently indexed by $k$.
The bounds are relevant when $d$ is sufficiently large.
Specifically, if  $d=(k+1)^2$ then it is easy to check that $k+d^{\frac{k}{k+1}}\le d$ and, by the monotonicity of the function $d\to d-d^{\frac{k}{k+1}}$, Theorem~\ref{thm:k_gen} applies
whenever $d\ge (k+1)^2.$ To simplify the discussion, we neglect the factors depending only on $k$ (which we did not try to optimize), and we also assume that $d$ is sufficiently large.
We observe, that the lower bound from Part (ii) of Theorem~\ref{thm:k_gen} approaches \eqref{eq:BC_lower} as $k$ tends to infinity.
However, the interval where $\eps$ must lie in gets smaller when $k$ increases.

\medskip

Finally, let us discuss the interplay between the lower bounds of Theorem~\ref{thm:k_gen} for different values of $k$. Again, we neglect the factors $c_k$, which were not optimized.
Observe that, for a fixed $k$, we get a lower bound that changes its nature at $\eps=d^{-\frac{k^2}{k+1}}$.
And, for a fixed $k=k_0$, this bound is better than the bounds for other values of $k$ if $d^{-k_0}\le \eps < d^{-(k_0-1)}$,
i.e., when the value $d^{-\frac{k^2}{k+1}}$ lies in this interval. %, i.e., the breaking point for $\eps$ between Part (i) and Part (ii),

Therefore, for $d$ large enough, we get the best lower bound from Theorem \ref{thm:k_gen} in the following way:
\begin{enumerate}
\item Choose the integer $k$ such that $\eps \in \left[d^{-k}, d^{-(k-1)}\right)$.
\item If $\eps \ge d^{-\frac{k^2}{k+1}}$ then $N(\eps,d)\ge c_k \cdot \added{\eps^{-\frac{k+1}{k}}}$ \deleted{$\eps^{-\frac{k}{k+1}}$} by Part (i) for $k$.
\item Otherwise $\eps < d^{-\frac{k^2}{k+1}}$, thus $N(\eps,d)\ge c_k \cdot d^{\frac{k}{k+1}}\eps^{-1}$ by Part (ii) for $k$.
\end{enumerate}

\subsection{Notation}
For a finite set $X$, we denote its size by $|X|$
\deleted{.With a slight abuse of notation, we also write $|A|$ for $A\subseteq [0,1]^d$ to denote the Lebesgue measure of $A$.
It will be, however, always clear from the context whether the argument of $|\cdot|$ is a finite set or not.}
\added{and} $\added{\vol(A)}$ \added{stands for the Lebesgue measure of $A\subset[0,1]^d$.}
For a positive integer $d$, we denote by $[d]$ the set $\{1,2,\ldots,d\}$.
Additionally, for a non-negative integer $k$, we write $\binom{[d]}k$ to denote the collection of all the $k$-element subsets of $[d]$.
Given a point $x \in [0,1]^d$ and an integer $i \in [d]$, we denote by $(x)_i$ the $i$-th coordinate of $x$.

\section{Part (i) of Theorem~\ref{thm:k_gen} and cover-free families}\label{sec:2}

We utilize a strong connection between lower bounds on $N(\eps,d)$ and a certain problem in extremal set theory.
The following is a generalization of the notion of an \emph{$r$-cover-free family}, a crucial notion in the proof of~\eqref{eq:TVV_2} in~\cite{TVV}.

\begin{define}
Let $\mathcal{F}$ be a family of subsets of a ground set $X$.
We say that $\mathcal{F}$ is \emph{$(k,r)$-cover-free} if no intersection of any $k$ sets of $\mathcal{F}$ is contained in the union of any other $r$ sets from $\mathcal{F}$,
i.e.,
if for all $A_1,\dots,A_k\in\mathcal{F}$ and all $B_1,\dots,B_r\in \mathcal{F}\setminus\{A_1,\dots,A_k\}$ it holds that
\[
\bigcap_{i=1}^k A_i \not\subset \bigcup_{j=1}^r B_j.
\]
\end{define}
Let $C(k,r,d)$ be the smallest size of the ground set such that a $d$-element $(k,r)$-cover-free family exists,
i.e.,
\[ C(k,r,d):= \min\{n \in \mathbb{N} \colon \exists (k,r)\text{-cover-free family } \mathcal{F}_n \text{ on }[n]\text{ with }|\mathcal{F}_n|=d\}.\]
Note that the case $k=1$ corresponds to $r$-cover-free families introduced in 1964 by Kautz and Singleton~\cite{KS}
and then intensively studied since the 1980s in various contexts~\cite{AB,AA, Engel,  E1,E2,F,Rusz, SWZ}.
The following very recent result of Michel and Scott~\cite{Michel_Scott} plays a crucial role for us.
\begin{theorem}[{\cite[Lemma 2.2]{Michel_Scott}}]
	\label{thm:michel_scott}
	If $k,s,t$ and $d$ are positive integers satisfying $d\ge k+t$, then $ C(k,s+t,d) \geq \frac{1}{2k^k}\cdot\min\bigl\{d^k, \ s(k+t)^k\bigr\}$.
\end{theorem}

The connection between lower bounds on $N(\eps,d)$ and cover-free families was first discovered in \cite{TVV}.
Our approach here is similar, but more elaborate and more flexible.
We define a very specific set of axis-parallel boxes of volume at least $\eps$,
which allows us to translate the bounds on $N(\eps,d)$ from below to the existence of $(k,r)$-cover-free families.
The role of $k$ in this reduction will be such that we obtain a whole range of new lower bounds,
each valid only for $\eps$ that is appropriately bounded from below.

Fix the dimension $d$.
For all positive integers $k$ and $\ell$ with $k+\ell\le d$ and every $u\in(0,1)$, we define a collection of boxes that have some $k$ sides equal to $(0,u)$ and some other $\ell$ sides equal to $(u,1)$.
Specifically, for a given set $K \subseteq [d]$ with $|K|=k$ and $L \subseteq [d]\setminus K$ with $|L|=\ell$, let $B^{K,L} \subseteq [0,1]^d$ be defined as
\[
B^{K,L}_u:=I_1 \times I_2 \times \cdots \times I_d, \quad \mbox{where }  \begin{cases}
	I_i =(0,u) \quad \mbox {for } i \in K, \\
	I_i =(u,1) \quad \mbox {for } i \in L, \\
	I_i =(0,1) \quad \mbox{for } i \in [d] \setminus \left(K \cup L\right).
\end{cases}
\]

Let $\cB(d,k,\ell,u) := \left\{ B^{K,L}_u \subseteq[0,1]^d: K \in \binom{[d]}{k} , L \in \binom{[d]\setminus K}{\ell} \right\}$.
We choose $\ell$ and $u$ based on the values of $k$ and $\eps$ to ensure that the boxes in $\cB(d,k,\ell,u)$ have volume at least $\eps$.
\begin{lemma}\label{lem:vol}
Let $k$ and $d$ be positive integers, let $\eps \in \left(0,2^{-k-2}\right)$, and set $u:=(4\eps)^{1/k}$.
If $d\ge k +\left\lfloor \frac1u\right\rfloor$, then $\added{\vol(B)}\text{\deleted{$|B|$}} > \eps$ for every $B \in \cB\left(d,k,\left\lfloor \frac1u\right\rfloor,u\right)$.
\end{lemma}
\begin{proof}
Observe that the bound on $\eps$ ensures that $u \in (0,1/2)$.
Therefore, it holds that $(1-u)^{\frac{1}{u}} > \frac{1}{4}$ by convexity of $2^{-2u}$ in this interval.
In particular,
\[
\added{\vol(B)}\text{\deleted{$|B|$}} = u^k \cdot (1-u)^{\left\lfloor \frac1u\right\rfloor} \ge 4\eps \cdot (1-u)^{\frac1{u}} > \eps
\,,\]
which finishes the proof.
\end{proof}

Let us now describe how we assign to a point set $X\subset [0,1]^d$ a certain family of subsets of $[d]$.
Given $u \in (0,1)$ and $j \in [d]$, we define
\[F^u_j := \{x \in X: (x)_j < u\}\,.\]
The connection between minimal dispersion and extremal set theory, which plays the central role in our proof of Part (i) of Theorem~\ref{thm:k_gen}, is the following.
\begin{lemma}\label{lem:CF}
Let $d$ be a positive integer, $u\in(0,1)$, $k$ and $\ell$ positive integers with $d\ge k+\ell$, and $X\subset [0,1]^d$.
If $X$ intersects every box in $\cB(d,k,\ell,u)$ then the~family $\mathcal{F}=\left\{F^u_1,F^u_2,\ldots,F^u_d\right\}$ is $(k,\ell)$-cover-free.
\end{lemma}
\begin{proof}
Suppose for the sake of contradiction that there exist two index sets $K \in \binom{[d]}k$ and $L \in \binom{[d]\setminus K}\ell$ such that
\begin{equation*}\label{eq:spor}
	\bigcap_{i\in K} F^u_i \subseteq \bigcup_{j \in L} F^u_j
\,.\end{equation*}
There must, however, exist a point $x \in X \cap B^{K,L}_u$.
Since $(x)_i < u$ for all $i \in K$, we get $x\in F^u_i$ for all $i\in K$ and hence $x\in \bigcap_{i\in K} F^u_i$.
On the other hand, $(x)_j > u$ for all $j\in L$ and thus $x\not\in \bigcup_{j \in L} F^u_j$. A contradiction.
\end{proof}

Lemmas~\ref{lem:vol} and~\ref{lem:CF} yield a lower bound on $N(\eps,d)$ in terms of $C(k,r,d)$.

\begin{cor}\label{cor:N_C}
If $k$ is a positive integer and $\eps \in \left(0,2^{-k-2}\right)$ such that they satisfy $d\ge k + \left\lfloor (4\eps)^{-1/k}\right\rfloor$,
then $N(\eps, d) \ge C\left(k, \left\lfloor (4\eps)^{-1/k} \right\rfloor, d\right).$
\end{cor}
\begin{proof}
Let $u:=(4\eps)^{1/k}$ and $\ell := \left\lfloor \frac1{u}\right\rfloor$.
If $X\subset [0,1]^d$ intersects all the axes-parallel boxes with volume larger than $\eps$,
then Lemma \ref{lem:vol} yields that $X$ intersects every box in $\cB\left(d,k,\ell,u\right)$.
By Lemma \ref{lem:CF}, the corresponding family $\mathcal{F}=\{F^u_1,\dots,F^u_d\}$ is $(k,\ell)$-cover-free.
In particular, the ground set of ${\mathcal F}$, which is $X$, contains at least $C(k,\ell,d)$ points.
\end{proof}

We are now ready to prove the main result of this section --- Part (i) of Theorem~\ref{thm:k_gen} --- which we restate here for the sake of convenience.

\begin{prop}\label{prop:part1}
Fix a positive integer $k$.
If $d$ is a positive integer satisfying $d\ge d^{\frac{k}{k+1}}+k$ and $\eps \in \left[d^{-\frac{k^2}{k+1}},2^{-k-2}\right)$, then
\begin{equation*}
N(\eps,d)\ge \frac{1}{16ek^k(k+1)} \cdot \eps^{-\frac{k+1}{k}}.
\end{equation*}
\end{prop}
\begin{proof}
The case $k=1$ follows from Theorem~\ref{thm:OLD}, so in the remaining we assume $k\ge2$.
First, observe that $k+\lfloor(4\eps)^{-1/k}\rfloor\le k+\eps^{-1/k}\le k+d^{\frac{k}{k+1}}\le d$.

Next, let $\ell :=  \left\lfloor (4\eps)^{-1/k} \right\rfloor$, and define %\footnote{$s\ge 1$ - clear; $t\ge 1$: We need $k\ge 2$ and $(4\eps)^{-1/k}\ge 2$.}
\[
s := \left\lceil \frac{\ell+1}{k+1}  \right\rceil
\quad\mbox{and}\quad
t := \left\lfloor \frac{k\cdot\ell-1}{k+1} \right\rfloor.
\]
Observe that this choice ensures that $s,t\ge 1$ and $s+t=\ell$.
Theorem \ref{thm:michel_scott} and Corollary~\ref{cor:N_C} yield that
\begin{align}
\notag N(\eps,d)&\ge C(k,\ell,d) = C(k,s+t,d)\\
\notag &\geq \frac{1}{2k^k}\cdot\min\biggl\{d^k, \frac{\ell+1}{k+1} \cdot \left(k+\left\lfloor \frac{k\cdot\ell-1}{k+1} \right\rfloor\right)^k \biggr\}\\
\notag %\label{eq:C_min}
&\geq \frac{1}{2k^k}\cdot\min\biggl\{d^k, \frac{(4\eps)^{-1/k}}{k+1}\left(\frac{k}{k+1} (4\eps)^{-1/k}\right)^k \biggr\}\\
\notag &\geq \frac{1}{2k^k}\cdot\min\biggl\{d^k, \frac{1}{e(k+1)}(4\eps)^{-\frac{k+1}k}\biggr\}\\
\notag &\ge \frac{1}{2k^ke(k+1)}\cdot\min\biggl\{d^k, (4\eps)^{-\frac{k+1}k}\biggr\}.
\end{align}
Under the condition on $\eps$, the minimum is attained by the second term, and since $4^{-\frac{k+1}k}\ge\frac18$ for $k\ge2$, the statement of the proposition follows.
\end{proof}

\section{Rescaling and Part (ii) of Theorem~\ref{thm:k_gen}}\label{sec:3}

We start with a rescaling-type observation analogous to \cite[Lemma 1]{AHR}, which was stated for $\disp^*(n,d)$.
\begin{lemma}\label{lem:k}
If $d$ and $b$ are positive integers and $\eps\in\big(0,\frac{1}{b}\big)$, then
\begin{equation*}%\label{eq:Nk}
	N(\eps,d) \geq b\cdot N(b\cdot\eps,d).
\end{equation*}
\end{lemma}

\begin{proof}
Let $\Omega:=[0,1]^d$ be the unit cube.
For a positive integer $n\in\Nat$ and an arbitrary box $A\subset\Omega$ we define in analogy to \eqref{eq:def_disp_nd} the minimal dispersion relative to $A$ as
\begin{equation*}
	\disp^*(n,d,A) := \inf_{\substack{X\subset A\\|X|=n}} \ \sup_{\substack{B:B\subset A \\ B\cap X=\emptyset}} \added{\vol(B)}\text{\deleted{$|B|$}}\,,
\end{equation*}
and, similarly to \eqref{eq:def_N}, its inverse function
\begin{equation}\label{eq:def_Nr}
	N(\eps,d,A) := \min\{n\in\Nat: \disp^*(n,d,A)\le \varepsilon\}.
\end{equation}
Note that $N(\eps,d,\Omega)=N(\eps,d)$.

Next, \added{for $i\in\{1,\dots,b\}$} we define $\Omega_i=(\frac{i-1}{b},\frac{i}{b})\times (0,1)\times\dots\times (0,1)\subset \Omega$. 
Observe that, up to a set of zero measure, $\left(\Omega_1,\dots,\Omega_b\right)$ is a partition of $\Omega$ into $b$ disjoint boxes of equal volume.
By \eqref{eq:def_Nr}, we need at least $N(\eps,d,\Omega_i)$ points to intersect every box of volume $\varepsilon>0$ within $\Omega_i$ for each $i\in\{1,\dots,b\}$.
Therefore, for the entire cube, it holds
\begin{equation}\label{eq:N_add}
N(\eps,d,\Omega) \geq N(\eps,d,\Omega_1) + \dots + N(\eps,d,\Omega_b) = b \cdot N(\eps,d,\Omega_1).
\end{equation}
Finally, consider the mapping $(x_1,x_2,\dots,x_d)\to (b\cdot x_1,x_2,\dots,x_d)$.
Clearly, it maps $\Omega_1$ onto $\Omega$, preserves the cardinality of subsets and transforms boxes in $\Omega_1$ to boxes in~$\Omega$ so that their volume enlarges exactly $b$-times.
Therefore, $N(\eps,d,\Omega_1) = N(b\cdot \eps,d,b\cdot \Omega_1) = N(b\cdot \eps,d,\Omega)$, which, together with \eqref{eq:N_add}, finishes the proof.
\end{proof}

Lemma~\ref{lem:k} serves as a tool for extending the validity of the lower bound in Part (i) of Theorem~\ref{thm:k_gen} to the regime of $\eps$ in Part (ii).
Indeed, for the $k$-th bound of Part (ii), we apply the lemma to move in the range of parameters,
where the $k$-th bound of Proposition~\ref{prop:part1} applies.

\begin{prop}\label{prop:part2}
Fix a positive integer $k$.
If $d$ is a positive integer satisfying $d\ge d^{\frac{k}{k+1}}+k$ and $\eps \in \left(0, d^{-\frac{k^2}{k+1}}\right)$, then
\begin{equation*}%\label{eq:k_ext}
N(\eps,d)\ge \frac{1}{64ek^k(k+1)} \cdot \frac{d^{\frac{k}{k+1}}}{\eps}
\,.
\end{equation*}
\end{prop}

\begin{proof}
Set $b:=\left\lceil d^{-\frac{k^2}{k+1}} \cdot \eps^{-1} \right\rceil$.
Note that this choice ensures that \[d^{-\frac{k^2}{k+1}}\le b\eps\le d^{-\frac{k^2}{k+1}}+\eps\le 2d^{-\frac{k^2}{k+1}}\,.\]
Therefore, Lemma \ref{lem:k} and Proposition~\ref{prop:part1} readily yield that
\begin{align*}
N(\eps,d) &\ge b\cdot N(b\cdot\eps,d) \ge \frac{1}{16ek^k(k+1)} \cdot d^{-\frac{k^2}{k+1}}\eps^{-1} \cdot (b \eps)^{-\frac{k+1}{k}} \\
&\ge \frac{1}{64ek^k(k+1)} \cdot \frac{d^{\frac{k}{k+1}}}{\eps}
\,,
\end{align*}
and this finishes the proof.
\end{proof}

\end{document}